\documentclass [11pt,twoside,a4paper]{article}
\usepackage{amsfonts}
\usepackage{amsthm}
\usepackage{amsmath}
\usepackage{amstext}
\usepackage{amssymb}
\usepackage{mathrsfs}
\usepackage{amscd}
\usepackage{xypic}
\usepackage{epsf}
\usepackage{graphicx}
\usepackage{fancybox}
\usepackage{color}
\usepackage{fancyhdr}
\usepackage[hang,footnotesize]{caption2}

\usepackage[
  top=1.5cm,
  headheight=1cm,
  body={132mm,205mm},
  centering
 ]{geometry}
\usepackage{enumerate}
\allowdisplaybreaks[4]
\usepackage{epsfig}
\usepackage{epstopdf}

\def\mathcal{\mathscr}
\newfont{\aaa}{cmb10 at 19pt}
\newfont{\bbb}{cmb10 at 11pt}
\newtheorem{lemma}{Lemma}[]
\newtheorem{theorem}{Theorem}
\newtheorem{corollary}{Corollary}[]

\newtheorem{definition}{Definition}[]
\def\v1{\vspace{1mm}}

\def\leq{\leqslant}

\def\geq{\geqslant}

\newcommand{\beq}{\begin{equation}}
\newcommand{\eeq}{\end{equation}}
\newcommand{\bey}{\begin{eqnarray}}
\newcommand{\eey}{\end{eqnarray}}
\newcommand{\beyy}{\begin{eqnarray*}}
\newcommand{\eeyy}{\end{eqnarray*}}

\makeatletter
\def\@evenhead{
\vbox{\hbox to \textwidth {}{\hspace{0mm}{\footnotesize
\thepage}}{\hspace{9cm} {\footnotesize {Linlin Fu et al.}}}
\protect\vspace{1truemm}\relax \hrule depth0pt height0.15truemm
width\textwidth}}
\def\@evenfoot{}
\def\@oddhead{\vbox{\hbox to \textwidth
{{\hspace{0cm}{\footnotesize Stability of additive-quadratic functional equation}\hfill{\footnotesize
\thepage}}\hspace{0mm}}{} \protect\vspace{1truemm}\relax\hrule
depth0pt height0.15truemm width\textwidth}}
\def\@oddfoot{}
\makeatother

\begin{document}

\thispagestyle{empty}

\fancypagestyle{firststyle}
{
\renewcommand{\topmargin}{-9mm}
\fancyhead[lO,RE]{\footnotesize submit article \\
......\\[3mm]
}
\fancyhead[RO,LE]{\scriptsize \bf 
} \fancyfoot[CE,CO]{}}
\renewcommand{\headrulewidth}{0pt}
\setcounter{page}{1}
\qquad\\[5mm]

\thispagestyle{firststyle}

\noindent{\aaa{Ulam stability of an additive-quadratic functional equation in F-space and quasi-Banach spaces\\[2mm]

\noindent{\bbb Linlin Fu$^1,$\quad Qi Liu$^{1*},$\quad Yongjin Li$^1$}\\[-1mm]

\noindent\footnotesize{1\ \ Department of Mathematics,
Sun Yat-sen University, Guangzhou 510275, China\\
\\[6mm]
\vskip-2mm \noindent{\footnotesize\,} \vskip 4mm

\normalsize\noindent{\bbb Abstract}\quad By adopting the direct method and fixed point method, we prove that the Hyers-Ulam stability of the
following additive-quadratic functional equation
\begin{equation}\label{equ(0.1)}
	f(x+y, z+w)+f(x-y, z-w)-2 f(x, z)-2 f(x, w)=0
\end{equation}
in $\beta$-homogeneous $F$-spaces and quasi-Banach spaces.
There are some differences that we consider the target space with the $\beta$-homogeneous norm and quasi-norm. Overcoming the $\beta$-homogeneous norm and quasi-norm bottlenecks, we get some new results.

\vspace{0.3cm}

\footnotetext{ \\
\hspace*{5.8mm}Corresponding author: Qi Liu, E-mail:
 liuq325@mail2.sysu.edu.cn}

\noindent{\bbb Keywords}\quad Hyers-Ulam stability, fixed point method, functional equation, $F$-space, quasi-Banach space\\
{\bbb MSC}\quad 39B52, 47H10\\[0.4cm]

\noindent{\bbb 1\quad Introduction and preliminaries}\\[0.1cm]

The stability problem of functional equations originated from a question of Ulam \cite{SM} in 1940 , concerning the stability of group homomorphisms.

Let $\left(G_{1}, \cdot \right)$ be a group and let $\left(G_{2}, *\right)$ be a metric group with the metric $d(\cdot, \cdot) .$ Given $\delta>0,$ does there exist a $\varepsilon>0,$ such that if a mapping $h: G_{1} \rightarrow G_{2}$ satisfies the inequality $$d(h(x . y), h(x) *h(y))\leqslant\delta$$ for all $x, y \in G_{1},$ then there exists a homomorphism $H: G_{1} \rightarrow G_{2}$ with $$d(h(x), H(x))\leqslant\varepsilon$$ for all $x \in G_{1} ?$

In $1941,$ Hyers \cite{DH} gave the first affirmative answer to the question of Ulam for Banach spaces. Let $f: E \rightarrow E^{\prime}$ be a mapping between Banach spaces such that
$$
\|f(x+y)-f(x)-f(y)\| \leqslant \delta
$$
for all $x, y \in E,$ and for some $\delta>0 .$ Then there exists a unique additive mapping $T: E \rightarrow E^{\prime}$ such that
$$
\|f(x)-T(x)\| \leqslant \delta
$$
for all $x \in E .$ In 1978 , Rassias \cite{TH} proved the following theorem.

\begin{theorem}\label{theorem 1.1}

\rm{\cite{TH}} Let $f: E \rightarrow E^{\prime}$ be a mapping from a normed vector space $E$ into a Banach space $E^{\prime}$ subject to the inequality
\begin{equation}\label{equ(1.1)}
	\|f(x+y)-f(x)-f(y)\| \leqslant \varepsilon\left(\|x\|^{p}+\|y\|^{p}\right)
\end{equation}
for all $x, y \in E,$ where $\varepsilon$ and $p$ are constants with $\varepsilon>0$ and $p<1 .$ Then there exists a unique additive mapping $T: E \rightarrow E^{\prime}$ such that
\begin{equation}\label{equ(1.2)}
	\|f(x)-T(x)\| \leqslant \frac{2 \varepsilon}{2-2^{p}}\|x\|^{p}
\end{equation}
for all $x \in E .$ If $p<0$ then inequality (\ref{equ(1.1)}) holds for all $x, y \neq 0,$ and (\ref{equ(1.2)}) for $x \neq 0 .$ Also, if the function $t \mapsto f(t x)$ from $\mathbb{R}$ into $E^{\prime}$ is continuous in $t \in \mathbb{R}$ for each fixed $x \in E,$ then $T$ is $\mathbb{R}$-linear.
\end{theorem}

Although stability problems have been studied successfully in the framework of Banach spaces, there are not many relevant results in $F$-spaces. One of the most important reasons is that the nonlinear structure of infinite-dimensional $F$-spaces and the failure of triangle inequality bring us challenges and difficulties. Besides these, for $F$-spaces, several results can be consulted in \cite{AF,KN} and the references therein. For more information about quasi-Banach spaces, the readers can refer to \cite{CW,RM,TA}.
Various more results for the stability of functional equations in quasi-Banach spaces can
be seen in  \cite{NVD2018,CP2008}.

Gil\'{a}nyi \cite{AG} showed that if $f$ satisfies the functional inequality
\begin{equation}\label{equ(1.3)}
	\|2 f(x)+2 f(y)-f(x-y)\| \leqslant\|f(x+y)\|
\end{equation}
then $f$ satisfies the Jordan-von Neumann functional equation
$$
2 f(x)+2 f(y)=f(x+y)+f(x-y).
$$
Fechner \cite{WF} and Gilányi \cite{AG} proved the Hyers-Ulam stability of the functional inequality (\ref{equ(1.3)}). The stability problems of functional equations and functional inequalities have been studied extensively by many authors (see \cite{IE,CP}).

Fixed point theory play an important role in functional analysis and other applied disciplines.
Next, we recall a fundamental result in fixed point theory.

\begin{theorem}\label{theorem 1.3}
\rm{
\cite{JD} Let $(X, d)$ be a complete generalized metric space and let $J: X \rightarrow X$ be a strictly contractive mapping with Lipschitz constant $\alpha<1 .$ Then for each given element $x \in X$, either
$$
d\left(J^{n} x, J^{n+1} x\right)=\infty
$$
for all nonnegative integers $n$ or there exists a positive integer $n_{0}$ such that

(1) $d\left(J^{n} x, J^{n+1} x\right)<\infty, \quad \forall n \geqslant n_{0}$;

(2) the sequence $\left\{J^{n} x\right\}$ converges to a fixed point $y^{*}$ of $J$;

(3) $y^{*}$ is the unique fixed point of $J$ in the set $Y=\left\{y \in X \mid d\left(J^{n_{0}} x, y\right)<\infty\right\}$;

(4) $d\left(y, y^{*}\right) \leqslant \frac{1}{1-\alpha} d(y, J y)$ for all $y \in Y$.}
\end{theorem}
By using the new fixed point method, the stability problems  of functional equations has been further studied extensively(see \cite{LCO,LCF,IE,CPF,VR}).

\begin{definition}\label{dt1}\rm{
	Consider $X$ be a linear space. A non-negative valued function $\|\cdot\|$ achieves
	an $F$-norm if satisfies the following conditions:

	{(1)} $\|x\|= 0$ if and only if $x=0$;

	{(2)} $\|\lambda x\|=\|x\|$ for all  $\lambda$,  $|\lambda|=1$;

	{(3)} $\|x+y\|\leq\|x\|+\|y\|$ for all $x,y\in X$;

	{(4)} $\|\lambda_{n}x\|\rightarrow 0$ provided  $\lambda_{n}\rightarrow 0$;

	{(5)} $\|\lambda x_{n}\|\rightarrow 0$ provided  $x_{n}\rightarrow 0$;

	{(6)} $\|\lambda_{n} x_{n}\|\rightarrow 0$ provided  $\lambda_{n}\rightarrow 0, x_{n}\rightarrow 0$.}

\end{definition}
{\it Then $(X,\|\cdot\|)$ is called an $F^\ast$-space. An F-space is a complete $F^\ast$-space}.

An $F$-norm  is called $\beta$-homogeneous $(\beta>0)$ if $\|tx\|=|t|^{\beta}\|x\|$ for all $x\in X$ and all $t\in \mathcal{C}$ (see \cite{RS,WA}).

If a quasi-norm  is $p$-subadditive, then it is called $p$-norm $(0 < p < 1)$.
In other words, if it satisfies
$$\Vert x+y\Vert^p\leq \Vert x\Vert^p+\Vert y\Vert^p,~~x,y\in X.$$
We note that the $p$-subadditive quasi-norm  $\Vert \cdot\Vert $ induces an $F$-norm.
We refer the reader to  \cite{KNP} and  \cite{AF} for background on it.

\begin{definition}\rm{\cite{NJNT}
	A quasi–norm on $\Vert \cdot\Vert$ on vector space $X$ over a field $K(
	\mathbb{R})$ is a map $X\longrightarrow[0, \infty)$ with the following properties:

	${\text { (1) }\|x\|=0 \text { if and only if } x=0}$

	${\text { (2) }\|a x\|=|a|\|x\|,~~ a \in \mathbb{R}, x \in X}$

	${\text { (3) }\|x+y\| \leq C(\|x\|+\|y\|),}~~ {x, y \in X}$

	where $C \geq 1$ is a constant independent of $x, y \in X$. The smallest $C$ for which
	$(3)$ holds in the definition is called the quasi-norm constant of $(X,\Vert \cdot\Vert)$.}
\end{definition}

It is vital to emphasize the well-known theorem in nonlocally convex theory, that is,
Aoki–Rolewicz theorem \cite{RS}, which asserts that for some $0 < p\leq 1$, every quasi-norm
admits  an equivalent $p$-norm.

 In this paper, we study the stability of the additive-quadratic functional equation (\ref{equ(0.1)}), which is closely related to the results by  Inho Hwang  and Choonkil Park  in 2020 \cite{IH}. There are some differences that we consider the target space with the $\beta$-homogeneous norm and quasi-norm. Overcoming the $\beta$-homogeneous norm and quasi-norm bottlenecks, we get some new results.

This paper is organized as follows: In Section 2, we prove the Hyers-Ulam stability of the additive-quadratic functional equation (\ref{equ(0.1)}) in $F$-spaces and quasi-Banach spaces by using the direct method. In Section $3,$ we prove the Hyers-Ulam stability of the additive-quadratic functional equation (\ref{equ(0.1)}) in $F$-spaces and quasi-Banach spaces by using the fixed point method.

\noindent \\[4mm]
\noindent{\bbb 2\quad Hyers-Ulam stability of the additive-quadratic functional equation (\ref{equ(0.1)}):	direct method}\\[0.1cm]

In this section,  we study the additive-quadratic functional equation (\ref{equ(0.1)}) in  $F$-spaces. The following lemma  plays a major role in our article.
\begin{lemma}\label{lemma 2.1}
	\rm{\cite{IH} Let $X,Y$ be vector spaces. If a mapping $f: X^{2} \rightarrow Y$ satisfies $f(0, z)=f(x, 0)=0$ and
	\begin{equation}\label{equ(2.1)}
			f(x+y, z+w)+f(x-y, z-w)-2 f(x, z)-2 f(x, w)=0
	\end{equation}
	for all $x, y, z, w \in X,$ then $f: X^{2} \rightarrow Y$ is additive in the first variable and quadratic in the second variable.}
\end{lemma}

Note that if $f: X \rightarrow Y$ satisfies $(\ref{equ(2.1)}),$ then the mapping $f: X \rightarrow Y$ is called an \emph{additive-quadratic} mapping.

\begin{theorem}\label{theorem 2.2}
	Let $X,Y$ be $\beta$-homogeneous $F$-spaces and $\varphi: X^{2} \rightarrow[0, \infty)$ be a function satisfying
	\begin{equation}\label{equ(2.2)}
		\Phi(x, y):=\sum_{j=1}^{\infty} 4^{(j-1)\beta} \varphi\left(\frac{x}{2^j}, \frac{y}{2^{j}}\right)<\infty
	\end{equation}
	for all $x, y \in X$ and $f: X^{2} \rightarrow Y$ be a mapping satisfying $f(x, 0)=f(0, z)=0$ and
	\begin{equation}\label{equ(2.3)}
		\|f(x+y, z+w)+f(x-y, z-w)-2 f(x, z)-2 f(x, w)\| \leqslant \varphi(x, y) \varphi(z, w)
	\end{equation}
	for all $x, y, z, w \in X .$ Then there exists a unique additive-quadratic mapping $F: X^{2} \rightarrow Y$ such that
	$$
	\|f(x, z)-F(x, z)\| \leqslant \min \left\{\Psi(x, x) \varphi(z, 0), \varphi(x, 0) \Phi(z, z)\right\}
	$$
	for all $x, z \in X,$ where
	$$
	\Psi(x, y):=\sum_{j=1}^{\infty} 2^{(j-1)\beta} \varphi\left(\frac{x}{2^{j}}, \frac{y}{2^{j}}\right)
	$$
	for all $x, y \in X$.
\end{theorem}
\begin{proof}
		${\bf  Step ~1}$
	Setting  $w=0$ and  $y=x$
	in $(\ref{equ(2.3)}),$ we can obtain
	\begin{equation}\label{equ(2.4)}
		\|f(2 x, z)-2 f(x, z)\| \leqslant \varphi(x, x) \varphi(z, 0).
	\end{equation}
	This means that
	$$
	\left\|f(x, z)-2 f\left(\frac{x}{2}, z\right)\right\| \leqslant \varphi\left(\frac{x}{2}, \frac{x}{2}\right) \varphi(z, 0)
	$$
	for all $x, z \in X .$ Hence
	\begin{align}
		\left\|2^{l} f\left(\frac{x}{2^{l}}, z\right)-2^{m} f\left(\frac{x}{2^{m}}, z\right)\right\|
		& \leqslant \sum_{j=l}^{m-1}2^{j\beta}\left\| f\left(\frac{x}{2^{j}}, z\right)-2 f\left(\frac{x}{2^{j+1}}, z\right)\right\| \label{equ(2.5)} \\
		& \leqslant \sum_{j=l}^{m-1} 2^{j\beta}\varphi\left(\frac{x}{2^{j+1}}, \frac{x}{2^{j+1}}\right) \varphi(z, 0)   \notag
	\end{align}
	for all nonnegative integers $m$ and $l$ with $m>l$ and all $x, z \in X .$

Applying (\ref{equ(2.5)}), we can deduce that the sequence $\left\{2^{k} f\left(\frac{x}{2^{k}}, z\right)\right\}$ is Cauchy for all $x, z \in X .$ Since $Y$ is a $F$-space, the sequence $\left\{2^{k} f\left(\frac{x}{2^{k}}, z\right)\right\}$ converges.  Thus, we  can define the mapping $P: X^{2} \rightarrow Y$ by
	$$
	P(x, z):=\lim _{k \rightarrow \infty} 2^{k} f\left(\frac{x}{2^{k}}, z\right)
	$$
	for all $x, z \in X$. Moreover, letting $l=0$ and passing the limit $m \rightarrow \infty$ in $(\ref{equ(2.5)}),$ we get
	\begin{equation}\label{equ(2.6)}
		\|f(x, z)-P(x, z)\| \leqslant \Psi(x, x) \varphi(z, 0)
	\end{equation}
	for all $x, z \in X$.

	On the other hand,
	it follows from (\ref{equ(2.2)}) and (\ref{equ(2.3)}) that
	$$
	\begin{aligned}
	&\|P(x+y, z+w)+P(x-y, z-w)-2 P(x, z)-2 P(x, w)\| \\
	=& \lim _{n \rightarrow \infty}\left\|2^{n}\left(f\left(\frac{x+y}{2^{n}}, z+w\right)+f\left(\frac{x-y}{2^{n}}, z-w\right)
	-2 f\left(\frac{x}{2^{n}}, z\right)-2 f\left(\frac{x}{2^{n}}, w\right)\right)\right\| \\
	\leqslant & \lim _{n \rightarrow \infty} 2^{n\beta} \varphi\left(\frac{x}{2^{n}}, \frac{y}{2^{n}}\right) \varphi(z, w)=0
	\end{aligned}
	$$
	for all $x, y, z, w \in X .$

	Then we can deduce that
	$$
	P(x+y, z+w)+P(x-y, z-w)-2 P(x, z)-2 P(x, w)=0
	$$
	for all $x, y, z, w \in X .$ By Lemma \ref{lemma 2.1}, the mapping $P: X^{2} \rightarrow Y$ is additive in the first variable and quadratic in second variable.

		${\bf  Step ~2}$
	Now, let $T: X^{2} \rightarrow Y$ be another additive-quadratic mapping satisfying $(\ref{equ(2.6)}) .$ Then we have
	$$
	\begin{aligned}
	\|P(x, z)-T(x, z)\| &=\left\|2^{q} P\left(\frac{x}{2^{q}}, z\right)-2^{q} T\left(\frac{x}{2^{q}}, z\right)\right\| \\
	& \leqslant\left\|2^{q} P\left(\frac{x}{2^{q}}, z\right)-2^{q} f\left(\frac{x}{2^{q}}, z\right)\right\|+\left\|2^{q} T\left(\frac{x}{2^{q}}, z\right)-2^{q} f\left(\frac{x}{2^{q}}, z\right)\right\| \\
	& \leqslant 2^{q\beta+1} \Psi\left(\frac{x}{2^{q}}, \frac{x}{2^{q}}\right) \varphi(z, 0) \\
	& =2\sum_{j=q}^{\infty} 2^{j\beta} \varphi(z, 0) \varphi\left(\frac{x}{2^{j+1}}  , \frac{x}{2^{j+1}}\right),
	\end{aligned}
	$$
	which tends to zero as $q \rightarrow \infty$ for all $x, z \in X$.
	Hence,
	$P(x, z)=T(x, z)$ for all $x, z \in X $.

	${\bf  Step ~3}$
Letting $y=0$ and $w=z$ in $(\ref{equ(2.3)}),$ we get
	\begin{equation}\label{equ(2.7)}
		\|f(x, 2 z)-4 f(x, z)\| \leqslant \varphi(x, 0) \varphi(z, z)
	\end{equation}
	and so
	$$
	\left\|f(x, z)-4 f\left(x, \frac{z}{2}\right)\right\| \leqslant \varphi(x, 0) \varphi\left(\frac{z}{2}, \frac{z}{2}\right)
	$$
	for all $x, z \in X .$ Hence
	\begin{align}
	\left\|4^{l} f\left(x, \frac{z}{2^{l}}\right)-4^{m} f\left(x, \frac{z}{2^{m}}\right)\right\| & \leqslant \sum_{j=l}^{m-1}  4^{j\beta}  \left\| f\left(x, \frac{z}{2^{j}}\right)-4 f\left(x, \frac{z}{2^{j+1}}\right)\right\| \label{equ(2.8)}\\
	& \leqslant \sum_{j=l}^{m-1} 4^{j\beta} \varphi(x, 0) \varphi\left(\frac{z}{2^{j+1}}, \frac{z}{2^{j+1}}\right)  \notag
	\end{align}
	for all nonnegative integers $m$ and $l$ with $m>l$ and all $x, z \in X .$ It follows from (\ref{equ(2.8)}) that the sequence $\left\{4^{k} f\left(x, \frac{z}{2^{k}}\right)\right\}$ is Cauchy for all $x, z \in X .$ Since $Y$ is a $F$-space, the sequence $\left\{4^{k} f\left(x, \frac{z}{2^{k}}\right\}\right.$ converges.
	Thus, we  can define the mapping $Q: X^{2} \rightarrow Y$ by
	$$
	Q(x, z):=\lim _{k \rightarrow \infty} 4^{k} f\left(x, \frac{z}{2^{k}}\right)
	$$
	for all $x, z \in X .$ Moreover, letting $l=0$ and passing the limit $m \rightarrow \infty$ in $(\ref{equ(2.8)}),$ we get
	\begin{equation}\label{equ(2.9)}
		\|f(x, z)-Q(x, z)\| \leqslant \varphi(x, 0) \Phi(z, z)
	\end{equation}
	for all $x, z \in X .$

	It follows from (\ref{equ(2.2)}) and (\ref{equ(2.3)}) that
	$$
	\begin{aligned}
	&\|Q(x+y, z+w)+Q(x-y, z-w)-2 Q(x, z)-2 Q(x, w)\| \\
	=& \lim _{n \rightarrow \infty}\left\|4^{n}\left(f\left(x+y, \frac{z+w}{2^{n}}\right)+f\left(x-y, \frac{z-w}{2^{n}}\right)-2 f\left(x, \frac{z}{2^{n}}\right)-2 f\left(x, \frac{w}{2^{n}}\right)\right)\right\| \\
	\leqslant
	& \lim _{n \rightarrow \infty} 4^{n\beta} \varphi(x, y) \varphi\left(\frac{z}{2^{n}}, \frac{w}{2^{n}}\right)=0
	\end{aligned}
	$$
	for all $x, y, z, w \in X .$ So
	$$
	Q(x+y, z+w)+Q(x-y, z-w)-2 Q(x, z)-2 Q(x, w)=0
	$$
	for all $x, y, z, w \in X .$ By Lemma $\ref{lemma 2.1},$ the mapping $Q: X^{2} \rightarrow Y$ is additive in the first variable and quadratic in second variable.

		${\bf  Step ~4}$
	Now, let $T: X^{2} \rightarrow Y$ be another additive-quadratic mapping satisfying $(\ref{equ(2.9)}) .$ Then we have
	$$
	\begin{aligned}
	\|Q(x, z)-T(x, z)\| &=\left\|4^{q} Q\left(x, \frac{z}{2^{q}}\right)-4^{q} T\left(x, \frac{z}{2^{q}}\right)\right\| \\
	& \leqslant\left\|4^{q} Q\left(x, \frac{z}{2^{q}}\right)-4^{q} f\left(x, \frac{z}{2^{q}}\right)\right\|+\left\|4^{q} T\left(x, \frac{z}{2^{q}}\right)-4^{q} f\left(x, \frac{z}{2^{q}}\right)\right\| \\
	& \leqslant 2\cdot  \varphi(x, 0) \sum_{j=q}^{\infty} 4^{q\beta} \varphi\left(\frac{z}{2^{q}}, \frac{z}{2^{q}}\right)
	\end{aligned}
	$$
	which tends to zero as $q \rightarrow \infty$ for all $x, z \in X .$ So we can conclude that $Q(x, z)=T(x, z)$ for all $x, z \in X$. This proves the uniqueness of $Q$.

		${\bf  Step ~5}$
	It follows from (\ref{equ(2.9)}) that
	$$
	2^{n\beta}\left\|f\left(\frac{x}{2^{n}}, z\right)-Q\left(\frac{x}{2^{n}}, z\right)\right\| \leqslant 2^{n\beta} \varphi\left(\frac{x}{2^{n}}, 0\right) \Phi(z, z)
	$$
	which tends to zero as $n \rightarrow \infty$ for all $x, z \in X .$ Since $Q: X^{2} \rightarrow Y$ is additive in the first variable, we get $\|P(x, z)-Q(x, z)\|=0$.  This means that  $F(x, z):=P(x, z)=Q(x, z)$ for all $x, z \in X .$ Thus there is an additive-quadratic mapping $F: X^{2} \rightarrow Y$ such that
	$$
	\|f(x, z)-F(x, z)\| \leqslant \min \left\{ \Psi(x, x) \varphi(z, 0),  \varphi(x, 0) \Phi(z, z)\right\}
	$$
	for all $x, z \in X$, as desired.

\end{proof}
\begin{corollary}
	Let $X, Y$ be quasi-Banach space and $\varphi: X^{2} \rightarrow[0, \infty)$ be a function satisfying
	$$
	\Phi(x, y):=\sum_{j=1}^{\infty} 4^{(j-1)p} \varphi\left(\frac{x}{2^j}, \frac{y}{2^{j}}\right)<\infty
	$$
	for all $x, y \in X$ and $f: X^{2} \rightarrow Y$ be a mapping satisfying $f(x, 0)=f(0, z)=0$ and
	\begin{equation} \label{equ (qusi-banach spaces)}
		\|f(x+y, z+w)+f(x-y, z-w)-2 f(x, z)-2 f(x, w)\| \leqslant \varphi^{\frac{1}{p}}(x, y) \varphi^{\frac{1}{p}}(z, w)
	\end{equation}
	for all $x, y, z, w \in X .$ Then there exists a unique additive-quadratic mapping $F: X^{2} \rightarrow Y$ such that
	$$
	\|f(x, z)-F(x, z)\| \leqslant \min \left\{\Psi^{\frac{1}{p}}(x, x) \varphi^{\frac{1}{p}}(z, 0), \varphi^{\frac{1}{p}}(x, 0) \Phi^{\frac{1}{p}}(z, z)\right\}
	$$
	for all $x, z \in X,$ where
	$$
	\Psi(x, y):=\sum_{j=1}^{\infty} 2^{(j-1)p} \varphi\left(\frac{x}{2^{j}}, \frac{y}{2^{j}}\right)
	$$
	for all $x, y \in X$.
\end{corollary}
\begin{proof}
	Let $\|\cdot\|_p=\|\cdot\|^p$, then it is obviously that $(Y,\|\cdot\|_p)$ is $p$-homogeneous $F$-space, so we can easily obtain the result from Theorem \ref{theorem 2.2}.

\end{proof}
\begin{corollary}
	Let $X, Y$ be $\beta$-homogenous $F$-spaces, $r>2$ and $\theta$ be nonnegative real numbers and $f: X^{2} \rightarrow Y$ be a mapping satisfying $f(x, 0)=f(0, z)=0$ and
	\begin{align}
	&\|f(x+y, z+w)+f(x-y, z-w)-2 f(x, z)-2 f(x, w)\|  \notag    \\
	&\leq\theta(\|x\|^r+\|y\|^r)(\|z\|^r+\|w\|^r)  \label{equ(2.10)}
	\end{align}
	for all $x, y, z, w \in X .$ Then there exists a unique additive-quadratic mapping $F: X^{2} \rightarrow Y$ such that
	$$
	\|f(x, z)-F(x, z)\| \leqslant \frac{2 \theta}{2^{\beta r}-2^\beta}\|x\|^{r}\|z\|^{r}
	$$
	for all $x, z \in X$.
\end{corollary}

\begin{proof}
	The proof follows from Theorem \ref{theorem 2.2} by taking $\varphi(x, y)=\sqrt{\theta}\left(\|x\|^{r}+\|y\|^{r}\right)$ for all $x, y \in X,$ since
	$$
	\min \left\{\frac{2 \theta}{2^{\beta r}-2^\beta}\|x\|^{r}\|z\|^{r}, \frac{2 \theta}{2^{\beta r}-2^{2\beta}}\|x\|^{r}\|z\|^{r}\right\}=\frac{2 \theta}{2^{\beta r}-2^\beta}\|x\|^{r}\|z\|^{r}
	$$
	for all $x, z \in$ $X$.
\end{proof}
\begin{theorem} \label{theorem 2.4}
\rm{	Let $X,Y$ be $\beta$-homogenous $F$-spaces, $\varphi: X^{2} \rightarrow[0, \infty)$ be a function satisfying
	\begin{equation}\label{equ(2.11)}
		\Psi(x, y):=\sum_{j=0}^{\infty} \frac{1}{2^{(j+1)\beta}} \varphi\left(2^{j} x, 2^{j} y\right)<\infty
	\end{equation}
	for all $x, y \in X$ and let $f: X^{2} \rightarrow Y$ be a mapping satisfying $f(x, 0)=f(0, z)=0$ and (\ref{equ(2.3)}) for all $x, z \in X .$ Then there exists a unique additive-quadratic mapping $F: X^{2} \rightarrow Y$ such that
	\begin{equation}\label{equ(2.12)}
		\|f(x, z)-F(x, z)\| \leqslant \min \left\{ \Psi(x, x) \varphi(z, 0),  \varphi(x, 0) \Phi(z, z)\right\}
	\end{equation}
	for all $x, z \in X,$ where
	$$
	\Phi(x, y):=\sum_{j=0}^{\infty} \frac{1}{4^{(j+1)\beta}} \varphi\left(2^{j} x, 2^{j} y\right)
	$$
	for all $x, y \in X$.}
\end{theorem}
\begin{proof}
	It follows from (\ref{equ(2.4)}) that
	$$
	\left\|f(x, z)-\frac{1}{2} f(2 x, z)\right\| \leqslant \frac{1}{2^\beta} \varphi(x, x) \varphi(z, 0)
	$$
	for all $x, z \in X .$ Hence
	\begin{align}
	\left\|\frac{1}{2^{l}} f\left(2^{l} x, z\right)-\frac{1}{2^{m}} f\left(2^{m} x, z\right)\right\| & \leqslant \sum_{j=l}^{m-1}\left\|\frac{1}{2^{j}} f\left(2^{j} x, z\right)-\frac{1}{2^{j+1}} f\left(2^{j+1} x, z\right)\right\| \label{equ(2.13)}\\
	& \leqslant \sum_{j=l}^{m-1} \frac{1}{2^{(j+1)\beta}} \varphi\left(2^{j} x, 2^{j} x\right) \varphi(z, 0) \notag
	\end{align}
	for all nonnegative integers $m$ and $l$ with $m>l$ and all $x, z \in X .$ It follows from (\ref{equ(2.13)}) that the sequence $\left\{\frac{1}{2^{k}} f\left(2^{k} x, z\right)\right\}$ is Cauchy for all $x, z \in X .$ Since $Y$ is a $\beta$-homogeneous $F$-space, the sequence $\left\{\frac{1}{2^{k}} f\left(2^{k} x, z\right)\right\}$ converges. So one can define the mapping $P: X^{2} \rightarrow Y$ by
	$$
	P(x, z):=\lim _{k \rightarrow \infty} \frac{1}{2^{k}} f\left(2^{k} x, z\right)
	$$
	for all $x, z \in X .$ Moreover, letting $l=0$ and passing the limit $m \rightarrow \infty$ in $(\ref{equ(2.13)}),$ we get
	\begin{equation}\label{equ(2.14)}
		\|f(x, z)-P(x, z)\| \leqslant \Psi(x, x) \varphi(z, 0)
	\end{equation}
	for all $x, z \in X.$

	It follows from (\ref{equ(2.3)}) and (\ref{equ(2.11)}) that
	$$
	\begin{aligned}
	&\|P(x+y, z+w)+P(x-y, z-w)-2 P(x, z)-2 P(x, w)\| \\
	=& \lim _{n \rightarrow \infty}\left\|\frac{1}{2^{n}}\left(f\left(2^{n}(x+y), z+w\right)+f\left(2^{n}(x-y), z-w\right)-2 f\left(2^{n} x, z\right)-2 f\left(2^{n} x, w\right)\right)\right\| \\
	\leqslant & \lim _{n \rightarrow \infty} \frac{1}{2^{n\beta}} \varphi\left(2^{n} x, 2^{n} y\right) \varphi(z, w)=0
	\end{aligned}
	$$
	for all $x, y, z, w \in X .$ So
	$$
	P(x+y, z+w)+P(x-y, z-w)-2 P(x, z)-2 P(x, w)=0
	$$
	for all $x, y, z, w \in X .$ By Lemma $\ref{lemma 2.1},$ the mapping $P: X^{2} \rightarrow Y$ is additive in the first variable and quadratic in second variable.

	Now, let $T: X^{2} \rightarrow Y$ be another additive-quadratic mapping satisfying (\ref{equ(2.14)}). Then we have
	$$
	\begin{aligned}
	\|P(x, z)-T(x, z)\| &=\left\|\frac{1}{2^{q}} P\left(2^{q} x, z\right)-\frac{1}{2^{q}} T\left(2^{q} x, z\right)\right\| \\
	& \leqslant\left\|\frac{1}{2^{q}} P\left(2^{q} x, z\right)-\frac{1}{2^{q}} f\left(2^{q} x, z\right)\right\|+\left\|\frac{1}{2^{q}} T\left(2^{q} x, z\right)-\frac{1}{2^{q}} f\left(2^{q} x, z\right)\right\| \\
	& \leqslant \frac{2}{2^{q\beta}} \Psi\left(2^{q} x, 2^{q} x\right) \varphi(z, 0)
	\end{aligned}
	$$
	which tends to zero as $q \rightarrow \infty$ for all $x, z \in X .$ So we can conclude that $P(x, z)=T(x, z)$ for all $x, z \in X .$ This proves the uniqueness of $P$.

	It follows from (\ref{equ(2.15)}) that
	\begin{equation}\label{equ(2.15)}
		\|f(x, 2 z)-4 f(x, z)\| \leqslant \varphi(x, 0) \varphi(z, z)
	\end{equation}
	and so
	$$
	\left\|f(x, z)-\frac{1}{4} f(x, 2 z)\right\| \leqslant \frac{1}{4^\beta} \varphi(x, 0) \varphi(2 z, 2 z)
	$$
	for all $x, z \in X .$ Hence
	\begin{align}
	\left\|\frac{1}{4^{l}} f\left(x, 2^{l} z\right)-\frac{1}{4^{m}} f\left(x, 2^{m} z\right)\right\| & \leqslant \sum_{j=l}^{m-1}\left\|\frac{1}{4^{j}} f\left(x, 2^{j} z\right)-\frac{1}{4^{j+1}} f\left(x, 2^{j+1} z\right)\right\| \label{equ(2.16)}  \\
	& \leqslant \sum_{j=l}^{m-1} \frac{1}{4^{(j+1)\beta}} \varphi(x, 0) \varphi\left(2^{j} z, 2^{j} z\right)  \notag
	\end{align}
	for all nonnegative integers $m$ and $l$ with $m>l$ and all $x, z \in X .$ It follows from (\ref{equ(2.16)}) that the sequence $\left\{\frac{1}{4^{k}} f\left(x, 2^{k} z\right)\right\}$ is Cauchy for all $x, z \in X .$ Since $Y$ is a $F$-space, the sequence $\left\{\frac{1}{4^{k}} f\left(x, 2^{k} z\right\}\right.$ converges. So one can define the mapping $Q: X^{2} \rightarrow Y$ by
	$$
	Q(x, z):=\lim _{k \rightarrow \infty} \frac{1}{4^{k}} f\left(x, 2^{k} z\right)
	$$
	for all $x, z \in X .$ Moreover, letting $l=0$ and passing the limit $m \rightarrow \infty$ in $(\ref{equ(2.16)}),$ we get
	\begin{equation} \label{equ(2.17)}
		\|f(x, z)-Q(x, z)\| \leqslant \varphi(x, 0) \Phi(z, z)
	\end{equation}
	for all $x, z \in X$.

	 It follows from (\ref{equ(2.3)}) and (\ref{equ(2.11)}) that
	$$
	\begin{aligned}
	&\|Q(x+y, z+w)+Q(x-y, z-w)-2 Q(x, z)-2 Q(x, w)\| \\
	=& \lim _{n \rightarrow \infty}\left\|\frac{1}{4^{n}}\left(f\left(x+y, 2^{n}(z+w)\right)+f\left(x-y, 2^{n}(z-w)\right)-2 f\left(x, 2^{n} z\right)-2 f\left(x, 2^{n} w\right)\right)\right\| \\
	\leqslant & \lim _{n \rightarrow \infty} \frac{1}{4^{n\beta}} \varphi(x, y) \varphi\left(2^{n} z, 2^{n} w\right)=0
	\end{aligned}
	$$
	for all $x, y, z, w \in X .$ So
	$$
	Q(x+y, z+w)+Q(x-y, z-w)-2 Q(x, z)-2 Q(x, w)=0
	$$
	for all $x, y, z, w \in X .$ By Lemma $\ref{lemma 2.1},$ the mapping $Q: X^{2} \rightarrow Y$ is additive in the first variable and quadratic in second variable.

	Now, let $T: X^{2} \rightarrow Y$ be another additive-quadratic mapping satisfying (\ref{equ(2.17)}). Then we have
	$$
	\begin{aligned}
	\|Q(x, z)-T(x, z)\| &=\left\|\frac{1}{4^{q}} Q\left(x, 2^{q} z\right)-\frac{1}{4^{q}} T\left(x, 2^{q} z\right)\right\| \\
	& \leqslant\left\|\frac{1}{4^q} Q\left(x, 2^{q} z\right)-\frac{1}{4^q} f\left(x, 2^{q} z\right)\right\|+\left\|\frac{1}{4^q} T\left(x, 2^{q} z\right)-\frac{1}{4^q} f\left(x, 2^{q} z\right)\right\| \\
	& \leqslant \frac{2}{4^{q\beta}} \varphi(x, 0) \Phi\left(2^{q} z, 2^{q} z\right),
	\end{aligned}
	$$
	which tends to zero as $q \rightarrow \infty$ for all $x, z \in X .$ So we can conclude that $Q(x, z)=T(x, z)$ for all $x, z \in X .$ This proves the uniqueness of $Q .$

	It follows from (\ref{equ(2.17)}) that
	$$
	\frac{1}{2^{n\beta}}\left\|f\left(2^{n} x, z\right)-Q\left(2^{n} z, z\right)\right\| \leqslant \frac{1}{2^{n\beta}} \varphi\left(2^{n} x, 0\right) \Phi(z, z)
	$$
	which tends to zero as $n \rightarrow \infty$ for all $x, z \in X .$ Since $Q: X^{2} \rightarrow Y$ is additive in the first variable, we get $\|P(x, z)-Q(x, z)\|=0,$ i.e., $F(x, z):=P(x, z)=Q(x, z)$ for all $x, z \in X .$ Thus there is an additive-quadratic mapping $F: X^{2} \rightarrow Y$ such that
	$$
	\|f(x, z)-F(x, z)\| \leqslant \min \left\{ \Psi(x, x) \varphi(z, 0),  \varphi(x, 0) \Phi(z, z)\right\}
	$$
	for all $x, z \in X .$
\end{proof}

\begin{corollary}
	Let $X$ be a quasi-Banach space, $\varphi: X^{2} \rightarrow[0, \infty)$ be a function satisfying
	$$
	\Psi(x, y):=\sum_{j=1}^{\infty} \frac{1}{2^{(j-1)p}} \varphi\left(\frac{x}{2^{j}}, \frac{y}{2^{j}}\right)<\infty
	$$
	for all $x, y \in X$ and $f: X^{2} \rightarrow Y$ be a mapping satisfying $f(x, 0)=f(0, z)=0$ and  (\ref{equ (qusi-banach spaces)})
	for all $x, y, z, w \in X .$ Then there exists a unique additive-quadratic mapping $F: X^{2} \rightarrow Y$ such that
	$$
	\|f(x, z)-F(x, z)\| \leqslant \min \left\{\Psi^{\frac{1}{p}}(x, x) \varphi^{\frac{1}{p}}(z, 0), \varphi^{\frac{1}{p}}(x, 0) \Phi^{\frac{1}{p}}(z, z)\right\}
	$$
	for all $x, z \in X,$ where
	$$
	\Phi(x, y):=\sum_{j=1}^{\infty} 4^{(j-1)p} \varphi\left(\frac{x}{2^j}, \frac{y}{2^{j}}\right)
	$$
	for all $x, y \in X$.
\end{corollary}
\begin{proof}
	Let $\|\cdot\|_p=\|\cdot\|^p$, then it is obviously that $(Y,\|\cdot\|_p)$ is $p$-homogeneous $F$-space, so we can easily obtain the result from Theorem \ref{theorem 2.4}.

\end{proof}

\begin{corollary}
	Let $X,Y$ be $\beta$-homogeous $F$-spaces, $r<1$ and $\theta$ be nonnegative real numbers and $f: X^{2} \rightarrow Y$ be a mapping satisfying (\ref{equ(2.10)}) and $f(x, 0)=f(0, z)=0$ for all $x, z \in X .$ Then there exists a unique additive-quadratic mapping $F: X^{2} \rightarrow Y$ such that
	$$
	\|f(x, z)-F(x, z)\| \leqslant \frac{2 \theta}{4^\beta-2^{r\beta}}\|x\|^{r}\|z\|^{r}
	$$
	for all $x, z \in X$.
\end{corollary}
\begin{proof}
	 The proof follows from Theorem \ref{theorem 2.4} by taking $\varphi(x, y)=\sqrt{\theta}\left(\|x\|^{r}+\|y\|^{r}\right)$ for all $x, y \in X,$ since $\min \left\{\frac{2 \theta}{2^\beta-2^{r\beta}}\|x\|^{r}\|z\|^{r}, \frac{2 \theta}{4^\beta-2^{r\beta}}\|x\|^{r}\|z\|^{r}\right\}=\frac{2 \theta}{4^\beta-2^{r\beta}}\|x\|^{r}\|z\|^{r}$ for all $x, z \in
	X$.
\end{proof}

\noindent \\[4mm]
\noindent{\bbb 3\quad Hyers-Ulam stability of the additive-quadratic functional equation (\ref{equ(0.1)}): fixed point method}\\[0.1cm]

Using the fixed point method, we prove the Hyers-Ulam stability of the additive quadratic functional equation (\ref{equ(0.1)}) in complex $F$-spaces.
\begin{theorem}\label{theorem 3.1}
	Let $X,Y$ be $\beta$-homogeous $F$-spaces, $\varphi: X^{2} \rightarrow[0, \infty)$ be a function such that there exists an $L<1$ with
	\begin{equation} \label{equ(3.1)}
		\varphi\left(\frac{x}{2}, \frac{y}{2}\right) \leqslant \frac{L}{4^\beta} \varphi(x, y) \leqslant \frac{L}{2^\beta} \varphi(x, y)
	\end{equation}
	for all $x, y \in X .$ Let $f: X^{2} \rightarrow Y$ be a mapping satisfying (\ref{equ(2.3)}) and $f(x, 0)=$ $f(0, z)=0$ for all $x, z \in X .$ Then there exists a unique additive-quadratic mapping $F: X^{2} \rightarrow Y$ such that
	\begin{equation} \label{equ(3.2)}
			\|f(x, z)-F(x, z)\| \leqslant \min \left\{\frac{L}{2^\beta(1-L)} \varphi(x, x) \varphi(z, 0), \frac{L}{4^\beta(1-L)} \varphi(x, 0) \varphi(z, z)\right\}
	\end{equation}
	for all $x, z \in X$.
\end{theorem}
\begin{proof}
	Letting $w=0$ and $y=x$ in $(\ref{equ(2.3)}),$ we get
	\begin{equation}\label{equ(3.3)}
		\|f(2 x, z)-2 f(x, z)\| \leqslant \varphi(x, x) \varphi(z, 0)
	\end{equation}
	for all $x, z \in X.$ Consider the set
	$$
	S:=\left\{h: X^{2} \rightarrow Y, \quad h(x, 0)=h(0, z)=0 \quad \forall x, z \in X\right\}
	$$
	and introduce the generalized metric on $S:$
	$$
	d(g, h)=\inf \left\{\mu \in \mathbb{R}_{+}:\|g(x, z)-h(x, z)\| \leqslant \mu \varphi(x, x) \varphi(z, 0), \quad \forall x, z \in X\right\}
	$$
	where, as usual, inf $\emptyset=+\infty .$ It is easy to show that $(S, d)$ is complete.

	Now we consider the linear mapping $J: S \rightarrow S$ such that
	$$
	J g(x, z):=2 g\left(\frac{x}{2}, z\right)
	$$
	for all $x, z \in X$. Let $g, h \in S$ be given such that $d(g, h)=\varepsilon$. Then
	$$
	\|g(x, z)-h(x, z)\| \leqslant \varepsilon \varphi(x, x) \varphi(z, 0)
	$$
	for all $x, z \in X .$ Hence
	$$
	\begin{aligned}
	\|J g(x, z)-J h(x, z)\| &=\left\|2 g\left(\frac{x}{2}, z\right)-2 h\left(\frac{x}{2}, z\right)\right\| \\&\leqslant 2^\beta \varepsilon \varphi\left(\frac{x}{2}, \frac{x}{2}\right) \varphi(z, 0) \\
	& \leqslant 2^\beta \varepsilon \frac{L}{2^\beta} \varphi(x, x) \varphi(z, 0)\\&=L \varepsilon \varphi(x, x) \varphi(z, 0)
	\end{aligned}
	$$
	for all $x, z \in X .$ So $d(g, h)=\varepsilon$ implies that $d(J g, J h) \leqslant L \varepsilon .$ This means that
	$$
	d(J g, J h) \leqslant L d(g, h)
	$$
	for all $g, h \in S$.

	It follows from (\ref{equ(3.3)}) that
	$$
	\begin{aligned}
	\left\|f(x, z)-2 f\left(\frac{x}{2}, z\right)\right\| &\leqslant \varphi\left(\frac{x}{2}, \frac{x}{2}\right) \varphi(z, 0) \\&\leqslant \frac{L}{2^\beta} \varphi(x, x) \varphi(z, 0)
	\end{aligned}
	$$
	for all $x, z \in X .$ So $d(f, J f) \leqslant \frac{L}{2^\beta}<\infty$.

	By Theorem $\ref{theorem 1.3},$ there exists a mapping $P: X^{2} \rightarrow Y$ satisfying the following:

	(i) $P$ is a fixed point of $J,$ i.e.,
	\begin{equation}\label{equ(3.4)}
		P(x, z)=2 P\left(\frac{x}{2}, z\right)
	\end{equation}
	for all $x, z \in X .$ The mapping $P$ is a unique fixed point of $J .$ This implies that $P$ is a unique mapping satisfying (\ref{equ(3.4)}) such that there exists a $\mu \in(0, \infty)$ satisfying
	$$
	\|f(x, z)-P(x, z)\| \leqslant \mu \varphi(x, x) \varphi(z, 0)
	$$
	for all $x, z \in X$.

	(ii) $d\left(J^{l} f, P\right) \rightarrow 0$ as $l \rightarrow \infty$. This implies the equality
	$$
	\lim _{l \rightarrow \infty} 2^{l} f\left(\frac{x}{2^{l}}, z\right)=P(x, z)
	$$
	for all $x, z \in X$.

	(iii) $d(f, P) \leqslant \frac{1}{1-L} d(f, J f),$ which implies
	$$
	\|f(x, z)-P(x, z)\| \leqslant \frac{L}{2^\beta(1-L)} \varphi(x, x) \varphi(z, 0)
	$$
	for all $x, z \in X$. By the same reasoning as in the proof of Theorem $\ref{theorem 2.4},$ one can show that the mapping $P: X^{2} \rightarrow Y$ is additive in the first variable and quadratic in the second variable.

	Letting $z=w$ and $y=0$ in $(\ref{equ(2.3)}),$ we get
	\begin{equation}\label{equ(3.5)}
		\|f(x, 2 z)-4 f(x, z)\| \leqslant \varphi(x, 0) \varphi(z, z)
	\end{equation}
	for all $x, z \in X$. Consider the set
	$$
	S:=\left\{h: X^{2} \rightarrow Y, \quad h(x, 0)=h(0, z)=0 \quad \forall x, z \in X\right\}
	$$
	and introduce the generalized metric on $S:$
	$$
	d^{\prime}(g, h)=\inf \left\{\mu \in \mathbb{R}_{+}:\|g(x, z)-h(x, z)\| \leqslant \mu \varphi(x, 0) \varphi(z, z), \forall x, z \in X\right\}
	$$
	where, as usual, inf $\phi=+\infty$. It is easy to show that $\left(S, d^{\prime}\right)$ is complete. Now we consider the linear mapping $J^{\prime}: S \rightarrow S$ such that
	$$
	J^{\prime} g(x, z):=4 g\left(\frac{x}{2}, z\right)
	$$
	for all $x, z \in X .$ Let $g, h \in S$ be given such that $d^{\prime}(g, h)=\varepsilon$. Then
	$$
	\|g(x, z)-h(x, z)\| \leqslant \varepsilon \varphi(x, 0) \varphi(z, z)
	$$
	for all $x, z \in X .$ Hence
	$$
	\begin{aligned}
	\left\|J^{\prime} g(x, z)-J^{\prime} h(x, z)\right\| &=\left\|4 g\left(x, \frac{z}{2}\right)-4 h\left(x, \frac{z}{2}\right)\right\| \leqslant 4^\beta \varepsilon \varphi(x, 0) \varphi\left(\frac{z}{2}, \frac{z}{2}\right) \\
	& \leqslant 4^\beta \varepsilon \frac{L}{4^\beta} \varphi(x, 0) \varphi(z, z)=L \varepsilon \varphi(x, 0) \varphi(z, z)
	\end{aligned}
	$$
	for all $x, z \in X .$ So $d^{\prime}(g, h)=\varepsilon$ implies that $d^{\prime}\left(J^{\prime} g, J^{\prime} h\right) \leqslant L \varepsilon .$ This means that
	$$
	d^{\prime}\left(J^{\prime} g, J^{\prime} h\right) \leqslant L d^{\prime}(g, h)
	$$
	for all $g, h \in S$.

	It follows from (\ref{equ(3.5)}) that
	$$
	\left\|f(x, z)-4 f\left(x, \frac{z}{2}\right)\right\| \leqslant \varphi(x, 0) \varphi\left(\frac{z}{2}, \frac{z}{2}\right) \leqslant \frac{L}{4^\beta} \varphi(x, 0) \varphi(z, z)
	$$
	for all $x, z \in X .$ So $d^{\prime}\left(f, J^{\prime} f\right) \leqslant \frac{L}{4^\beta}<\infty.$

	By Theorem $\ref{equ(1.3)},$ there exists a mapping $Q: X^{2} \rightarrow Y$ satisfying the following:

	(i) $Q$ is a fixed point of $J^{\prime},$ i.e.,
	\begin{equation}\label{equ(3.6)}
		Q(x, z)=4 Q\left(x, \frac{z}{2}\right)
	\end{equation}
	for all $x, z \in X .$ The mapping $Q$ is a unique fixed point of $J^{\prime} .$ This implies that $Q$ is a unique mapping satisfying (\ref{equ(3.6)}) such that there exists a $\mu \in(0, \infty)$ satisfying
	$$
	\|f(x, z)-Q(x, z)\| \leqslant \mu \varphi(x, 0) \varphi(z, z)
	$$
	for all $x, z \in X$.

	(ii) $d\left(J^{\prime l} f, Q\right) \rightarrow 0$ as $l \rightarrow \infty$. This implies the equality
	$$
	\lim _{l \rightarrow \infty} 4^{l} f\left(x, \frac{z}{2^{l}}\right)=Q(x, z)
	$$
	for all $x, z \in X$.

	(iii) $d(f, Q) \leqslant \frac{1}{1-L} d\left(f, J^{\prime} f\right),$ which implies
	$$
	\|f(x, z)-Q(x, z)\| \leqslant \frac{L}{4^\beta(1-L)} \varphi(x, 0) \varphi(z, z)
	$$
	for all $x, z \in X$
	By the same reasoning as in the proof of Theorem $\ref{theorem 2.2},$ one can show that the mapping $Q: X^{2} \rightarrow Y$ is additive in the first variable and quadratic in the second variable.

	By the same reasoning as in the proof of Theorem $\ref{theorem 2.2},$ we get $\|P(x, z)-Q(x, z)\|=$ $0,$ i.e., $F(x, z):=P(x, z)=Q(x, z)$ for all $x, z \in X .$ Thus there is an additive-quadratic mapping $F: X^{2} \rightarrow Y$ such that
	$$
	\|f(x, z)-F(x, z)\| \leqslant \min \left\{\frac{L}{2^\beta(1-L)} \varphi(x, x) \varphi(z, 0), \frac{L}{4^\beta(1-L)} \varphi(x, 0) \varphi(z, z)\right\}
	$$
	for all $x, z \in X$.
\end{proof}

\begin{corollary}
	Let $X,Y$ be $\beta$-homogeous $F$-spaces, $r>1$ and $\theta$ be nonnegative real numbers and $f: X^{2} \rightarrow Y$ be a mapping satisfying (\ref{equ(2.10)}) and $f(x, 0)=f(0, z)=0$ for all $x, z \in X .$ Then there exists a unique additive-quadratic mapping $F: X^{2} \rightarrow Y$ such that
	$$
	\|f(x, z)-F(x, z)\| \leqslant \frac{2 \theta}{2^{r\beta}-2^\beta}\|x\|^{r}\|z\|^{r}
	$$
	for all $x, z \in X$.
\end{corollary}
\begin{proof}
	The proof follows from Theorem \ref{theorem 3.1} by taking $L=\frac{2^{r\beta}-2^\beta}{2^{r\beta}-2^\beta+1}$ and $\varphi(x, y)=$ $\sqrt{\theta}\left(\|x\|^{r}+\|y\|^{r}\right)$ for all $x, y \in X,$ since
	$$
	\min \left\{\frac{2 \theta}{2^{(r-1)\beta}-1}\|x\|^{r}\|z\|^{r}, \frac{2 \theta}{2^{r\beta}-2^\beta}\|x\|^{r}\|z\|^{r}\right\}=\frac{2 \theta}{2^{r\beta}-2^\beta}\|x\|^{r}\|z\|^{r}
	$$
	for all $x, z \in X$.
\end{proof}

\begin{theorem}\label{theorem 3.3}
	Let $X,Y$ be $\beta$-homogeous $F$-spaces and $\varphi: X^{2} \rightarrow[0, \infty)$ be a function such that there exists an $L<1$ with
	\begin{equation}\label{equ(3.7)}
		\varphi(x, y) \leqslant 2^\beta L \varphi\left(\frac{x}{2}, \frac{y}{2}\right) \leqslant 4^\beta L \varphi\left(\frac{x}{2}, \frac{y}{2}\right)
	\end{equation}
	for all $x, y \in X .$ Let $f: X^{2} \rightarrow Y$ be a mapping satisfying satisfying (\ref{equ(2.3)}) and $f(x, 0)=$ $f(0, z)=0$ for all $x, z \in X .$ Then there exists a unique additive-quadratic mapping $F: X^{2} \rightarrow Y$ such that
	$$
	\|f(x, z)-F(x, z)\| \leqslant \min \left\{\frac{1}{2^\beta(1-L)} \varphi(x, x) \varphi(z, 0), \frac{1}{4^\beta(1-L)} \varphi(x, 0) \varphi(z, z)\right\}
	$$
	for all $x, z \in X$
\end{theorem}
\begin{proof}
	Consider the complete metric spaces $(S, d)$ and $\left(S, d^{\prime}\right)$ given in the proof of Theorem \ref{theorem 3.1} .

	Now we consider the linear mapping $J: S \rightarrow S$ such that
	$$
	J g(x, z):=\frac{1}{2} g(2 x, z)
	$$
	for all $x, z \in X .$ It follows from (\ref{equ(3.3)}) that
	$$
	\left\|f(x, z)-\frac{1}{2} f(2 x, z)\right\| \leqslant \frac{1}{2^\beta} \varphi(x, x) \varphi(z, 0)
	$$
	for all $x, z \in X .$ So $d(f, J f) \leqslant \frac{1}{2^\beta}$.

	By the same reasoning as in the proof of Theorem $\ref{theorem 2.2},$ one can show that there exists a unique additive-quadratic mapping $P: X^{2} \rightarrow Y$ such that
	$$
	\|f(x, z)-P(x, z)\| \leqslant \frac{1}{2^\beta(1-L)} \varphi(x, x) \varphi(z, 0)
	$$
	for all $x, z \in X$.

	By the same reasoning as in the proof of Theorem $\ref{theorem 2.2},$ one can show that the mapping $P: X^{2} \rightarrow Y$ is additive in the first variable and quadratic in the second variable.

	Now we consider the linear mapping $J^{\prime}: S \rightarrow S$ such that
	$$
	J^{\prime} g(x, z):=4 g\left(\frac{x}{2}, z\right)
	$$
	for all $x, z \in X.$ It follows from (\ref{equ(3.5)}) that
	$$
	\left\|f(x, z)-\frac{1}{4} f(x, 2 z)\right\| \leqslant \frac{1}{4^\beta} \varphi(x, 0) \varphi(z, z)
	$$
	for all $x, z \in X .$ So $d^{\prime}\left(f, J^{\prime} f\right) \leqslant \frac{1}{4^\beta}$.

	By the same reasoning as in the proof of Theorem $\ref{theorem 2.2},$ one can show that there exists a unique additive-quadratic mapping $Q: X^{2} \rightarrow Y$ such that
	$$
	\|f(x, z)-Q(x, z)\| \leqslant \frac{1}{4^\beta(1-L)} \varphi(x, 0) \varphi(z, z)
	$$
	for all $x, z \in X.$

	By the same reasoning as in the proof of Theorem $\ref{theorem 2.2},$ one can show that the mapping $Q: X^{2} \rightarrow Y$ is additive in the first variable and quadratic in the second variable.

	By the same reasoning as in the proof of Theorem $\ref{theorem 2.2},$ we get $\|P(x, z)-Q(x, z)\|=$ $0,$ i.e., $F(x, z):=P(x, z)=Q(x, z)$ for all $x, z \in X .$ Thus there is an additive-quadratic mapping $F: X^{2} \rightarrow Y$ such that
	$$
	\|f(x, z)-F(x, z)\| \leqslant \min \left\{\frac{1}{2^\beta(1-L)} \varphi(x, x) \varphi(z, 0), \frac{1}{4^\beta(1-L)} \varphi(x, 0) \varphi(z, z)\right\}
	$$
	for all $x, z \in X$.
\end{proof}
\begin{corollary}
	Let $X,Y$ be $\beta$-homogeous $F$-spaces, $r<1$ and $\theta$ be nonnegative real numbers and $f: X^{2} \rightarrow Y$ be a mapping satisfying (\ref{equ(2.10)}) and $f(x, 0)=f(0, z)=0$ for all $x, z \in X .$ Then there exists a unique additive-quadratic mapping $F: X^{2} \rightarrow Y$ such that
	$$
	\|f(x, z)-F(x, z)\| \leqslant \frac{2\theta}{4^\beta-2^{r\beta}}\|x\|^{r}\|z\|^{r}
	$$
	for all $x, z \in X .$
\end{corollary}
\begin{proof}
	The proof follows from Theorem \ref{theorem 3.3} by taking $L=2^{\beta(r-2)}$ and $\varphi(x, y)=$ $\sqrt{\theta}\left(\|x\|^{r}+\|y\|^{r}\right)$ for all $x, y \in X,$ since
	$$\min \left\{\frac{2 \theta}{4^\beta-2^{r\beta}}\|x\|^{r}\|z\|^{r}, \frac{2^{\beta+1}\theta}{4^\beta-2^{r\beta}}\|x\|^{r}\|z\|^{r}\right\}=\frac{2 \theta}{4^\beta-2^{r\beta}}\|x\|^{r}\|z\|^{r}$$
	for all $x, z \in X$.
\end{proof}

\noindent \\[4mm]

\noindent\bf{\footnotesize Acknowledgements}\quad\rm
{\footnotesize This work was supported by the National Natural Science Foundation of P. R. China (Nos. 11971493 and 12071491).}\\[4mm]


\begin{thebibliography}{99}
\bibitem{TA}  Aoki, T.  Locally bounded topological spaces, Proc. Imp. Acad. Tokyo, 18, (1942), 588–594.


\bibitem{AF}   Albiac, F.  {\it Nonlinear structure of some classical quasi-Banach spaces
	and $F$-spaces}, J. Math. Anal. Appl. 340 (2008), 1312–1325.


\bibitem{NJNT}  Bayoumi,  A.  {\it  Foundations of Complex Analysis in Non Locally Convex Spaces-Function Theory Without Convexity Conditions}. in:Mathematics Studied 193, North Holland, Amsterdam, New York, Tokyo (2003).


\bibitem{LCO}   Cadariu, L.;   Radu, V.  On the stability of the Cauchy functional equation: a fixed point approach, Grazer Math. Ber. $346(2004), 43-52 $.

\bibitem{LCF}   Cadariu, L.;   Radu, V.  Fired point methods for the generalized stability of functional equations in a single variable, Fixed Point Theory Appl. 2008 , Art. ID 749392 (2008).

\bibitem{JD}  Diaz,  J.;   Margolis, B.  A fixed point theorem of the altemative for contractions on a generalized complete metric space, Bull. Am. Math. Soc. $74(1968), 305-309 $.

\bibitem{NVD2018}   Dung, N.V.;  Hang,  V.T.L.  The generalized hyperstability of general linear equations in quasi-Banach spaces. J. Math. Anal. Appl. 462 (2018), 131– 147.



\bibitem{IE}  El-Fassi, I.Z.  Solution and approximation of radical quintic fumctional equation related to quintic mapping in quasi-$\beta$-Banach spaces, Rev. R. Acad. Cienc. Exactas Fís. Nat. Ser. A Mat. 113 ( 2019 ), no. $2,675-687$.



\bibitem{WF}  Fechner,  W. Stability of a functional inequalities associated with the Jordan-von Netwann fimctional equation, Aequationes Math. $71(2006), 149-161 $.



\bibitem{AG}   Gilanyi,  A.  On a problem by K. Nikodem, Math. Inequal. Appl. $5(2002), 707-710$.

\bibitem{DH} Hyers, D.H. On the stability of the linear functional equation, Proc. Nat. Acad. Sci. U.S.A. 27 $(1941), 222-224$.

\bibitem{IH} Hwang   I,  Park C. Ulam stability of an additive-quadratic functional equation in Banach spaces[J]. Journal of Mathematical Inequalities, 2020(2):421-436.

\bibitem{KNP}  Kalton, N.J.;   Peck, N.T. ; Rogers, J.W.  {\it An F-Space Sampler}, London Math. Lecture Notes, vol. 89.
Cambridge Univ. Press, Cambridge (1985).


\bibitem{KN}    Kalton,  N.J.{\it Curves with zero derivative in $F$-spaces}, Glasg. Math. J. 22 (1981), 19–29.


\bibitem{RM}  Malceski,  R.  Sharp triangle inequalities in quasi-normed spaces, British J. Math. Comput. Sci., 5,
(2015), 258–265.


\bibitem{CP} Park, C. Additive $\rho$ -functional inequalities and equations, J. Math. Inequal. $9(2015), 17-26$.



\bibitem{CPF} Park, C. Fixed point method for set-valued fimctional equations, J. Fixed Point Theory Appl. 19 $(2017), 2297-2308$.


\bibitem{CP2008}  Park, C.;   Rassias, T.M. Isometric additive mappings in generalized quasi-Banach spaces. Banach J Math Anal. 2 (2008), 59–69.

\bibitem{VR}  Radu, V. The fixed point alternative and the stability of functional equations, Fixed Point Theory 4 $(2003), 91-96$.

\bibitem{TH}  Rassias, T.M.  On the stability of the linear mapping in Banach spaces, Proc. Am. Math. Soc. 72 $(1978), 297-300$.

\bibitem{RS}  Rolewicz, S. {\it  Metric Linear Spaces}, PWN-Polish Sci. Publ, Reidel and Dordrecht. 1984.

\bibitem{SM} Ulam, S.M.  Problems of Modern Mathematics; Sciences Editions; John Wiley $\&$ Sons Inc.: New York, NY, USA, 1964.


\bibitem{CW}  Wu, C.;   Li, Y. On the triangle inequality in quasi-Banach spaces, J. Inequal. Pure Appl. Math.,
9, (2008).


\bibitem{WA}    Wilansky, A. {\it Modern Methods in Topological Vector Space}, McGraw-Hill International Book Co, New York. 1978.









\end{thebibliography}
\end{document}